\documentclass{amsart}
\usepackage{amsmath,amscd,amstext,amsthm,amsfonts,latexsym}
\usepackage{eufrak, graphicx,epstopdf,stmaryrd}
\usepackage{mathtools, abraces,yhmath, float, amssymb}
\usepackage{qtree, forest, tikz-qtree}
\usepackage{color, fancybox, empheq}
\usetikzlibrary{fit}
\usepackage{caption}

\DeclareCaptionType{InfoBox}
\usetikzlibrary{shapes}
\useforestlibrary{linguistics}
\usepackage{multirow, longtable, caption}

\theoremstyle{plain}
\newtheorem{theorem}{Theorem}[section]
\newtheorem{proposition}[theorem]{Proposition}
\newtheorem{lemma}[theorem]{Lemma}

\newtheorem{corollary}[theorem]{Corollary}
\newtheorem{maintheorem}{Theorem}

\newtheorem{maincoro}[maintheorem]{Corollary}

\newtheorem{definition}[theorem]{Definition}

\newtheorem{question}{Question}

\newtheorem{claim}{Claim}

\theoremstyle{definition}
\newtheorem{remark}{Remark}

\newcommand{\field}[1]{\mathbb{#1}}

\newcommand{\ZZ}{\field{Z}}
\newcommand{\NN}{\field{N}}

\newcommand{\FF}{\field{F}}
\def\N{\NN}
\def\ninf{n\to+\8}

\def\wh{\widehat}

\newcommand{\al} {\alpha}       
\newcommand{\be} {\beta}

\def\eps{\varepsilon}

\newcommand{\ka} {\kappa}

\newcommand{\om} {\omega}

\def\s{\sigma}
\renewcommand{\S}{\Sigma}
\def\disp{\displaystyle}
\def\FFF{\FF_{2}^{+}}

\def\8{\infty}
\def\wt{\widetilde}

\newcommand{\ol}[1]{\overline{#1}}

\newcommand {\CA}{{\mathcal A}}

\newcommand {\CF}{{\mathcal F}}

\newcommand {\CT}{{\mathcal T}}

\newcommand {\CX}{{\mathcal X}}
\newcommand {\CY}{{\mathcal Y}}

\newcommand {\GA}{{\mathfrak A}}
\newcommand {\GB}{{\mathfrak B}}

\newcommand {\GI}{{\mathfrak I}}
\newcommand {\GJ}{{\mathfrak J}}
\newcommand {\GK}{{\mathfrak K}}
\newcommand {\GL}{{\mathfrak L}}

\def\ST{\CA^{\FF_2^+}}

\newcommand{\comment}[1]{}

\newcommand{\tritree}[3]{
\begin{forest}
[{#1}[{#2}][{#3}]]
\end{forest}
}

\newcommand{\heptatree}[7]{
\begin{forest}
[{#1}
[{#2}[{#4}][{#5}]]
[{#3}[{#6}][{#7}]]
]
\end{forest}
}

\begin{document}

\title{The Jacaranda tree is strongly aperiodic and  has zero entropy}


\author{A. Baraviera}
\thanks{AB and RL want to thank CIRM for kind support for a recherche en binôme in October 2022.}
\address{Instituto de Matemática e Estatística - UFRGS\\
Avenida Bento Gonçalves, 9500  - Porto Alegre - RS - Brasil\\
CEP 91509-900}
\email{baravi@mat.ufrgs.br}

\author{Renaud Leplaideur}
\address{ISEA, Universit\'e de la Nouvelle-Cal\'edonie \& LMBA UMR6205}
\email{renaud.leplaideur@unc.nc}

\date{\today}

\begin{abstract}
We prove that the Jacaranda tree obtained as a fixed point for a substreetution in previous work of the authors is strongly aperiodic and that the number of patches increases linearly with respect to the size of the patch. As a consequence we get that the tree has zero entropy. 
 \end{abstract}

\maketitle

\section{Introduction}
\subsection{The Jacaranda tree and the substreetutions}
In the present paper, we put a step forward the ergodic study of the Jacaranda tree which is obtained from a special example of substreetution as defined by authors in \cite{bara-lep1}.

Substreetutions are substitutions acting on the set of colored binary trees  $\{0,1\}^{\FFF}$, where $\FFF$ is the free semi-group with two generators $a$ and $b$. They extend to $\{0,1\}^{\FFF}$ classical objects as the Thue-Morse and the Fibonacci substitutions on $\{0,1\}^{\N}$. 

In  \cite{bara-lep1} it was proved that the closure of $\FFF$-orbit of the fixed point for some special substreetution $H$, $X:=\overline{\{T_{\om}(\GJ),\ \om\in \FFF\}}$, is minimal and non-periodic, in the sense that it is not reduced to a periodic orbit. In other words, and following the terminology, it was proved that $X$ is \emph{weakly aperiodic}. 
We remind that the natural action of $\FFF$ is given by the two maps $T_{a}$ and $T_{b}$ which respectively send a binary (colored) tree $\GA$ on its $a$-follower or $b$-follower.

There are several motivations to study substitutions on $\{0,1\}^{\FFF}$. The principal one is a long work in progress that aims to export the thermodynamic formalism via transfer operator to higher dimensional group (or semi-group) action. One reason for that is to continue to better understand similitudes and differences between  Ergodic or Statistical Mechanics viewpoints for Thermodynamic formalism. Statistical Mechanics viewpoint usually deals with $\ZZ^{d}$-action or even $\FFF$-actions (see \cite{ising, Rozikov}). 

Ergodic viewpoint deals with transfer operators. For $\ZZ$-actions it links the thermodynamic quantities (such as pressure, Gibbs measures, etc) to the spectral properties of that operator. This has never be done for $\ZZ^{d}$-actions (with $d\ge 2$) and, this is probably due to the existence of  the ``natural'' orientation in $\ZZ$ which is the key point to define the transfer operator. This natural direction fails to exist in $\ZZ^{d}$. For that reason authors were naturally led to study $\FFF$-actions. Other works related to phase transitions  and quasi-periodic systems (see \cite{BLL, BL1,BL2}) also led authors to prospect for substitutions adapted to $\FFF$-actions, since their attractor are example of quasi-periodic systems.
 
Along the way, many other interesting questions arose, as by-products of the initial goal. 
For instance, a notion of Sturmian trees has been introduced and studied (see {\it e.g.} \cite{KLLS}). We remind that for the Fibonacci substitution the attractor is a Sturmian shift. It was thus quite natural to inquire to get example of such trees generated by substreetutions. 

The question of entropy is also of prime importance. Entropy measures the complexity of the system. For trees, the first natural question is to define the right normalization, since the increase is expected to be super exponential. Several notions of entropies with different normalizations have been defined  in the literature (see \cite{bufetov,Petersen-Salama-20,Ban-Chang}). 
We also point out that for usual $\N$-actions, entropy is also related to the number of preimages, and this quantity is crucial to properly define the transfer operator. 

\medskip
In the present paper we  continue the work in progress, advancing in the study of topological properties of the
free semi-group action on the  Jacaranda tree. We prove that the stabilizer is reduced to the empty word $\epsilon$, which means that $X$ is \emph{strongly aperiodic}. Concerning complexity, we prove that the number of patches essentially increases linearly with respect to the length. This is one path in the good direction to check if $\GJ$ is Sturmian. As a by-product, we get that $X$ has zero entropy whatever the definition we take.

\subsection{Main results}

\subsubsection{Aperiodicity}

We remind that for some group $G$ acting on some space $\CX$, $g\in G$ is a stabilizer for $x\in \CX$ if $g.x=x$. 

\begin{definition}
\label{def-perio}
We say that a $\FFF$-invariant set  is strongly periodic if it is finite.
We say that a $\FFF$-invariant set is weakly periodic if the set of stabilizers is non-empty.

\medskip
If a $\FFF$-invariant set is not strongly periodic then it is said to be weakly aperiodic. If it is not weakly periodic then it is said to be strongly aperiodic.
\end{definition}

In \cite{bara-lep1} it is proved that $X$ is weakly aperiodic. We prove here a stronger result:
\begin{maintheorem}\label{main-aperio}
$X$ is strongly aperiodic.
\end{maintheorem}

\begin{remark}
\label{rem-aperiodic}
We emphasize that being minimal does not prohibits the existence of $\GA\in X$ and $\om\in \FFF$, $|\om|\ge 1$ such that $T_{\om}(\GA)=\GA$. Hence being minimal and weakly aperiodic  does not imply that $X$ is strongly aperiodic. 
$\blacksquare$\end{remark}

\subsubsection{Complexity and topological entropy}

\begin{definition}
\label{def-kn}
A patch of size $n\ge 1$ is a finite binary tree with $n$ lines that appears in $\GJ$.
We denote by $K_{n}$ the set of patches of length $n$ in $X$. Its cardinal is $\kappa_{n}$.

For $C\in K_{n}$, $[C]$ denotes the set of $\GA\in X$ starting as $C$.
\end{definition}
We emphasize (see below for technical results on $X$) that for any $C\in K_{n}$ $[C]=B(\GA,2^{-n})$ for any $\GA$ in $[C]$.

\begin{maintheorem}\label{main-complexsubexpo}
There exists $1<C<+\8$ such that for every $n$,
$$n+2\le \kappa_{n}\le Cn+4.$$
\end{maintheorem}
This result has to be compared to the one for ``classical'' substitutions (see \cite{pansiot}). It is known that the equivalent $p(n)$ for $\ka_{n}$ is either in $O(n^{2})$, $O(n\log n)$, $O(n\log\log n)$, $O(n)$ or $O(1)$. However, we point out that that proof does not seem to be easily adaptable to substreetutions. It is also interesting to connect this result with the concept of quasi-Sturmian
trees introduced in \cite{KLLS};
those are the colored  trees where  $\kappa_n = n+c$ (in the particular case where $c=1$ the trees are called simply Sturmian)
for $n \geq N_0$.

Hence, a natural question is to inquire if in our case an equality $\kappa_n = n+c$ holds,  at least for large values of $n$.

\bigskip
As said above entropy measures the complexity of the system. It matters with how $\kappa_{n}$ increases. For trees, the main issue, at least for general set of trees, is to find the right normalization. Going in that direction, Petersen and Salama define the entropy for a colored binary tree (see \cite{Petersen-Salama-20}, ) as
 $$h_{PS}:=\limsup_{\ninf}\dfrac1{2^{n}}\log \ka_{n}.$$
This makes sense since the normalization factor $2^{2^{n}}$ is the cardinality of $\{0,1\}^{2^{n}}$ and is approximatively the number of patches of length $n$ in $\{0,1\}^{\FFF}$.  

In another direction Ban and Chang define entropy as $h_{BC}:=\limsup_{\ninf}\dfrac1{n}\log{ \log \ka_{n}}$ (see \cite{Ban-Chang}). 

An immediate consequence of Theorem \ref{main-complexsubexpo} is

\begin{maincoro}\label{main-entroPSzero}
The Petersen-Salama  and Ban-Chang entropies for $\GJ$ are equal to zero.
\end{maincoro}

\subsubsection{Entropy for a skew-product extension of $X$}

In order to define an entropy for a compact set $\CX$ of trees, Bufetov introduced
an  idea whose details are in section \ref{secbufetov-entropy}; he also related this entropy  to
the usual topological entropy  as follows (in our settings).

Let us set $\CY:= \{a, b\}^{\NN} \times \CX $  and  consider the skew-product $F:\CY \to \CY $ defined as
$$
  (\omega, x) \in  \{a, b\}^{\NN} \times \CX \mapsto F(\omega, x) = (\sigma(\omega), T_{\omega_0}(x)),
$$
 where $\sigma$ is the usual unilateral shift on $\{a, b\}^{\NN}$.
 Then, Bufetov  showed equality 
  $$h_{B}(\CT) :=  h_{top}(\CF) - \log{2}.$$

\medskip
In our setting we set $\CX=X$. Then we get:
\begin{maintheorem}\label{main-Fzeroentro}
The map $F$ has topological entropy equal to $\log2$.
\end{maintheorem}
  Hence, Bufetov's entropy is zero for the action of the free semi-group on the Jacaranda tree.

 \medskip
  We finish this subsection with an open question : 
  
 \begin{question}
 Is $(Y,F)$ an expansive dynamical system ?
\end{question}

We remind that expansiveness is a sufficient condition to get upper semi-continuity for the metric entropy. Our next step will be to check how invariant measures for $F$ may give better descriptions of $X$ or $\GJ$. For that purpose, studying thermodynamic formalism for $F$ seems a good way. Hence proving upper semi-continuity would be helpful. 

We remind that expansiveness means that for some $\eps>0$, if $d_{Y}(F^{n}(\om, \GA), F^{n}(\om',\GA'))<\eps$ for every $n\ge 0$, then $(\om,\GA)=(\om',\GA')$. It is immediate that $d_{Y}(F^{n}(\om, \GA), F^{n}(\om',\GA'))<1$ for every $n$ implies $\om=\om'$. On the other hand, $d_{Y}(F^{n}(\om, \GA), F^{n}(\om,\GA'))<\eps$  for every $n$ only means that $\GA$ and $\GA'$ do coincide along the \emph{enlarged} path $\om$. It is however not clear that this yields $\GA=\GA'$. 

Furthermore, for any $\GA$ in $X$, sites along the path $ b^{\8}$ are all equal to 0 except (may be) the root. This holds because any even line in $\GJ$ is a concatenation of 10 and any even line is a concatenation of 0010 and 0000. Similarly, for every even tree, the word along the path $(ab)^{\8}=ababab\ldots$ is  $\otimes(10)^{\8}$, where $\otimes$ is the root of the considered tree.  For odd tree it is $\otimes(01)^{\8}$. 

This shows that many trees do coincide along the two extremal paths but are different. Hence, some results go in the direction that $(Y,F)$ is expansive, some others go int eh opposite direction. This also shows that proving expansiveness is not immediate, nor non-expansiveness.




\subsection{Plan of the paper}
In Section \ref{sec-rappel} we remind some facts on substreetutions and the Jacaranda tree and we prove Theorem \ref{main-aperio}.
Section \ref{sec-patches} is devoted to give estimates for $\kappa_{n}$. This yields the proofs of Theorems \ref{main-complexsubexpo} and \ref{main-Fzeroentro}.


\section{Reminders on Substreetutions and proof that $\GJ$ is strongly aperiodic}\label{sec-rappel}.


\subsection{Reminders on Substreetutions}

\subsubsection{Binay colored trees and $\FFF$-action}

$\FFF$ is the free monoid with two generators, $a$ and $b$. It is the collection of finite words in $a$ and $b$. The empty word is denoted by $e$.
For $\om\in \FFF$, $|\om|$ is its length and denotes the number of letter that compose $\om$. By definition $|e|=0$.

The set of colored binary trees we consider is $\{0,1\}^{\FFF}$. All the trees we shal consider are these ones, and we will just refer to them as trees.
If $\GA$ is a tree and $\om$ is in $\FFF$, $\GA_{\om}$ is the digit at position $\om$. If $\GA$ is a tree $T_{a}(\GA)$ is the new tree obtain when considering the new root at site $a$ and thus forgetting the old root and the other part of the tree. Similarly $T_{b}(\GA)$ is the subtree with root at position $b$.

\bigskip
The distance between two trees $\GA$ and $\GB$ is $2^{-N(\GA,\GB)}$ where $N(\GA,\GB)$ is the minimal integer $n$ such that $\GA_{\om}\neq \GB_{\om}$ and $|\om|=n$.

In other words, $d(\GA,\GB)=2^{-n}$ means that $\GA$ and $\GB$ have different root if $n=0$, and $\GA_{\om}$ and $\GB_{\om}$ do coincide  for every $\om$, such that $|\om|\le n$ and for at least one $\om$ with $|\om|=n$,  one  of the followers of $\GA_{\om}$  is different to the same follower for $\GB_{\om}$.

Note that the space of trees $\ST$ is compact (for the metric we introduced) as a product of compact spaces. The subset of trees with root equal to 0 (or 1) is also compact as a closed set included into a compact set.

\bigskip
If $\GA$ is a binary tree,  and $\om=\om_{0}\ldots \om_{n}$ is in $\FFF$, we set  
$T_{\om}(\GA)=T_{\om_{n}}\circ\ldots \circ T_{\om_{1}}\circ T_{\om_{0}}(\GA)$. This corresponds to consider the substree in $\GA$ with root equal to the site $\om$ in $\GA$.

\subsubsection{Colored binary trees and substreetutions}

A {\bf substreetution}\footnote{Actually this is a \emph{constant length 2} substitution.} on trees is a map $H$ on the set of configurations defined by concatenation as follows:
\begin{enumerate}
\item $H$ maps each site to a truple (actually a root with two followers), the value depending only on the value of the digit at the site. See Figure \ref{Fig1-substi} with the box with dashline.
\item $H$ connects images of subtrees (followers) as indicated on Figure \ref{Fig1-substi}, with $\GI,\GJ,\GK,\GL\in \{H(\GA),H(\GB)\}$.
\end{enumerate}
 The order word $\GI \GJ \GK \GL$ is called the {\bf grammar} of the substreetution.

The substreetution is said to be \emph{marked} if $H(0)=\tritree{$i$}{}{}$ and $H(1)=\tritree{1-i}{}{}$, $i=0,1$.

\medskip
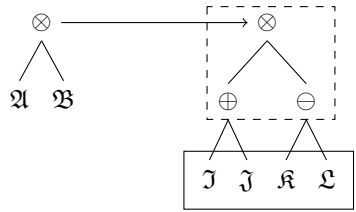
\begin{figure}[h]
\begin{tikzpicture}
\begin{scope}
\Tree [.\node (0) {$\otimes$} ; [.\node {$\GA$} ;  ] [.\node (r) {$\GB$}; ] ]
\end{scope}
\begin{scope}[xshift=3cm]
\Tree [.\node (1) {$\otimes$} ; [.\node(2) {$\oplus$} ; \node(a){$\GI$};  \node(b){$\GJ$};  ] [.\node (r) {$\ominus$}; \node(c){$\GK$}; \node(d){$\GL$}; ] ]
\node[draw, fit=(a) (b) (c) (d) ]{};
\end{scope}
\draw[->] (0)..controls +(east:2) and +(west:1)..(1);
\draw[dashed] (2.2,0.3) rectangle ++(1.7,-1.5);

\end{tikzpicture}
\caption{A market substreetution}
\label{Fig1-substi}
\end{figure}

\subsubsection{The Jacaranda tree}

The substreetution $H$ we consider here is the one  given by $\disp 0\mapsto$ \begin{forest}
[$0$[$1$][$0$]]
\end{forest} and $\disp 1\mapsto $\begin{forest}
[$1$[$1$][$0$]]
\end{forest}, equipped with the grammar BBAB.

There exists a unique fixed point $\GJ$ with root 0. It is called the \emph{Jacaranda tree}.
The closure of its orbit is a minimal dynamical system $X$ and is not periodic.
There also exists a unique fixed point $\GJ'$ with root 1. It coincides with $\GJ$ except at the root.

More precisely we have for $\otimes=0,1$: 

$$H\left(\tritree{$\otimes$}{$\GA$}{$\GB$}\right)=\heptatree{$\otimes$}{$1$}{$0$}{$H(\GB)$}{$H(\GB)$}{$H(\GA)$}{$H(\GB)$},$$

\subsubsection{Types, source map}

The map $\chi$ on words in $\{0,2\}^{2^{\N}}$ is defined by $\chi(10)=0010$ and $\chi(CD)=\chi(D)\chi(D)\chi(C)\chi(D)$ if $C$ and $D$ are  in $\{0,1\}^{2^{n}}$ (for any $n$ but the same $n$).

A line standing at an odd level (root is at line 0) in $\GJ$ is a concatenation of $10$'s. A line at level $2^{n}(2m+1)$ is a concatenation of $\chi^{n}(10)$.

If $\GA$ belongs to $X:=\ol{\{T_{\om}(\GJ),\om\in \FFF\}}$, there exists $\om_{k}$ with $|\om_{k}|\to+\8$ as $k\to+\8$ such that
$$\GA=\lim_{k\to+\8}T_{\om_{k}}(\GJ).$$
Then it turns out that there exists a unique $n$ such that for any sufficiently big $k$,
$$|\om_{k}|=2^{n}(2m_{k}+1).$$
If $n=0$ we say that $\GA$ is odd (or of odd-type). If $n\ge 1$ we that that $\GA$ is even, and more precisely we say that $\GA$ is of $2^{n}$-type.
Odd and even trees form a partition, and more generally trees of type $2^{n}$ with $n\ge 1$ form a partition of even trees.

\bigskip
If $\GA$ is of $2^{n}$-type with $n\ge 1$, there exists a unique $\GB$ such that $H(\GB)=\GA$. Furthermore, $\GB$ is of $2^{n-1}$-type. The map $\GA\to\GB$ is called the \emph{source map} and we write $\GB=s(\GA)$.

Note that the source can actually be defined on $\FFF$, since for every $\om\in \FFF$ there exists a unique $\wt s(\om)$ with length $|\om|/2$ such that
$$T_{\om}\circ H=H\circ T_{\wt{s}(\om)}$$

\subsection{Proof of Theorem \ref{main-aperio}}
Assume that there exists $\GA\in X$ and $\om\in \FFF$, $|\om|\ge 1$ such that $\GA=T_{\om}(\GA)$. Note that $|\om|$ must be an even integer since $T_{\om}(\GA)$ must have the same parity than $\GA$ (either odd or even).

\begin{claim}[Claim Cutting]\label{claim-cuttingaperio}
There exist $\GA' \in X$ and $\om'\in \FFF$ with $|\om'|=\dfrac{|\om|}{2}$ such that $T_{\om'}(\GA')=\GA'$.
\end{claim}
\begin{proof}[Proof of the claim]
If $\GA$ is even, then we set $\GA':=s(\GA)$ and there exists $\om'\in \FFF$ with $|\om'|=\dfrac{|\om|}{2}$ such that
$$H(\GA')=\GA=T_{\om}(\GA)=T_{\om}\circ H(\GA')=H(T_{\om'}(\GA')).$$

If $\GA$ is odd, set $\om=\om_{0}\ldots \om_{n}$, $\wh\om:=\om_{1}\ldots\om_{n}\om_{0}$. Then
$$T_{\om_{0}}(\GA)=T_{\om_{0}}T_{\om}(\GA)=T_{\wh\om}(T_{\om_{0}}(\GA)).$$
Furthermore, $T_{\om_{0}}(\GA)$ is even.  Hence we are sent to the previous point.
\end{proof}
Applying the claim, we get that $|\om'|$ must be an even integer. There exists some $k$ and $m$ such that $|\om|=2^{k}(2m+1)$.
Hence, applying $k$ times Claim \ref{claim-cuttingaperio}  we arrive to a contradiction: there exists $\GB\in X$ and $\om''$ with odd length such that 
$$\GB=T_{\om''}(\GB).$$

\section{Estimations for $\kappa_{n}$}\label{sec-patches}

\subsection{Bound from below  for $\kappa_{n}$ and beginning of proof of Theorem \ref{main-complexsubexpo}}

First, we state two lemmas that extend known-results for classical substitutions. 

\begin{lemma}
\label{lem-kappan}
$\kappa_{n+1}\ge \kappa_{n}$.
\end{lemma}
\begin{proof}
Each element $C$ in $K_{n}$ has a ``continuation'' $\wt C$ to be an element in $K_{n+1}$. Hence, two different $C$ and $C'$ in $K_{n}$ yields two different $\wt C$ and $\wt C'$ in $K_{n+1}$. This yields $\kappa_{n+1}\ge \ka_{n}$.
\end{proof}

\begin{lemma}
\label{lem-kanstatio}
Assume that $\ka_{n_{0}+1}=\ka_{n_{0}}$ for some $n_{0}$. Then, for every $n\ge n_{0}$, $\ka_{n}=\ka_{n_{0}}$.
\end{lemma}
\begin{proof}
Each $C\in K_{n_{0}}$ admits a unique continuation (on the bottom) to define an element $\wt C$ in $K_{n_{0}}$. This holds because otherwise, we would get $\ka_{n_{0}+1}>k_{n_{0}}$.
Set images $T_{a}(\wt C)$ and $T_{b}(\wt C)$ (see Fig. \ref{fig-kappastatio}) define two elements of $K_{n}$, respectively denoted by $C'$ and $C''$.

\begin{figure}[htbp]
\begin{center}
\includegraphics[scale=0.7]{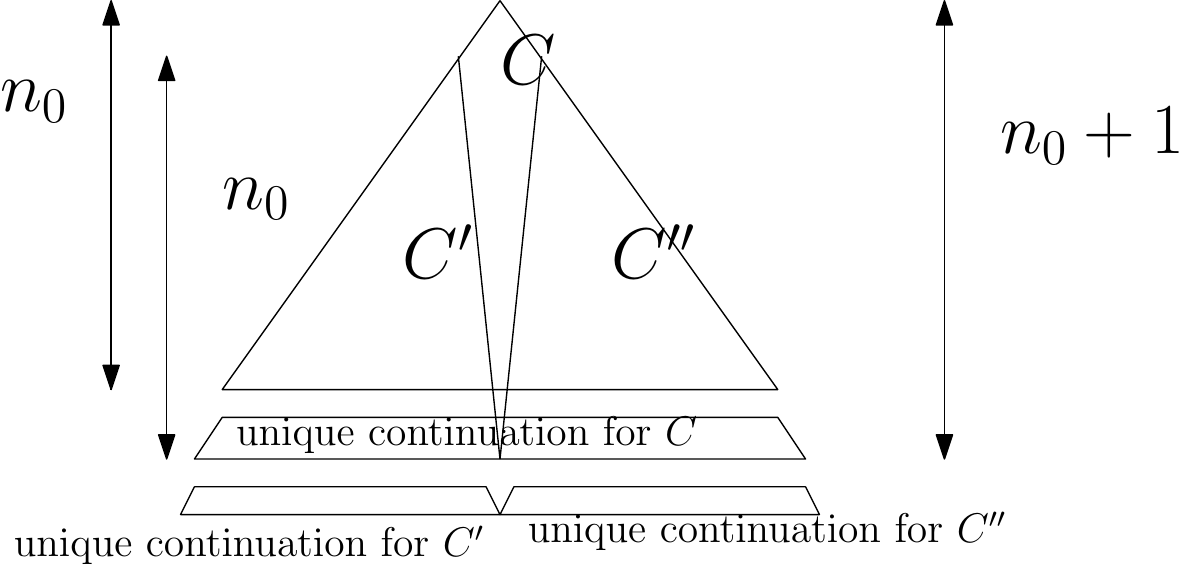}
\caption{Eventually stationarity for $\kappa_{n}$}
\label{fig-kappastatio}
\end{center}

\end{figure}

Again, $C'$ and $C''$ admit a unique continuation, which yields that $\wt C$ admits a unique continuation $\wh C$ in $K_{n_{0}+2}$. By induction we get $\ka_{n}=\ka_{n_{0}}$ for any $n\ge n_{0}$.
\end{proof}

\begin{proposition}
\label{prop-kappastrictmono}
The sequence $(k_{n})$ is increasing.
\end{proposition}
\begin{proof}
By Lemma \ref{lem-kappan} the sequence is non-decreasing. If for some $n_{0}$ $\ka_{n_{0}+1}=\ka_{n_{0}}$, then the sequence is stationary (Lemma \ref{lem-kanstatio}).

This yields that for any $C$ with length $n_{0}$, there exists a unique element in $X$ starting as $C$. Let us denote it by $\GA(C)$. This yields that the graph with vertices all the $\GA(C)$, $C\in K_{n_{0}}$ and arrows defined by  images by $T_{a}$ and $T_{b}$, is closed. Hence $\GJ$ is strongly pre-periodic, which is false.
\end{proof}

An immediate corollary is:

\begin{corollary}
\label{kan-lin}
For any $n$, $\ka_{n}\ge n+2$. Hence $\disp\liminf_{\ninf}\dfrac{\ka_{n}}{n}\ge 1$.
\end{corollary}

%
%
%

\subsection{Bound from above for $\kappa_{n}$ and end of the proof of Theorem \ref{main-complexsubexpo}}

\subsubsection{An inequality sastified by $(\kappa_{n})$}

\begin{proposition}
\label{prop-ka2n}
For any $n$, $\ka_{2n}\le \ka_{n}+\ka_{n+1}$.
\end{proposition}
\begin{proof}
Inequality is obvious if $n=1$.
Let $n\ge 2$ be in $\NN$. Let $C$ be an element in $K_{2n}$ and let $\GA\in X$ be in $[C]$. We know first four lines for $\GA$, hence we can determine if it is even or odd.

For simplicity we denote by $K_{n,o}$, $K_{n,e}$ the set of patches coinciding with odd or even trees. Their cardinality is respectively denoted by $\ka_{n,o}$ and $\ka_{n,e}$.

If $\GA$ is even, then  we set $\GB:=s(\GA)$. The first $2n$-lines  in $\GA$ are entirely defined by the first $n$ lines in $\GB$. Furthermore this definition is one-to-one.
This means that the number of patches $C$ in $K_{2n}$ of even type is equal to $\ka_{n}$, hence $\ka_{2n,e}=\ka_{n}$.

If $\GA$ is of odd type, we consider any preimage $\GA'$ of $\GA$. It belong to some (unique) $[C']$, with $C'\in K_{2n+2}$, and is even. Two different such $\GA$'s belonging to different $[C']$'s  yield two different $\GA'$'s.
This yields that $\ka_{2n,o}$  is lower or equal to $\ka_{2n+2,e}=\ka_{n+1}$.
\end{proof}

\subsubsection{Special sequences satisfying  that inequality}

We consider two numerical sequences $(u_{n})$ and $(v_{n})$ satisfying
\begin{enumerate}
\item $u_{2}=v_{2}=\al\ge 1$, $u_{3}=v_{3}=\be\ge \al+1$,
\item $\disp\forall n\ge 2, \begin{cases}
u_{2n}\le u_{n}+u_{n+1},\\
v_{2n}=v_{n}+v_{n+1}.
\end{cases}$
\item $(u_{n})$ and $(v_{n})$ are increasing
\item $\forall n\ge 2$, $v_{2n+1}= v_{2n+2}-1$.
\end{enumerate}

\begin{lemma}
\label{lem-unvn}
For every $n\ge 2$, $u_{n}\le v_{n}$.
\end{lemma}
\begin{proof}
The proof is done by induction. Inequality holds for $n=2$ and $n=3$. Let us assume it holds for every $k\le n$ and let us prove it also holds for every $k\le n+1$.

\medskip
$\bullet$ If $n+1$ is even, say $n+1=2k$, then $k\ge 2$. Hence
$$k=n+1-k\le n-1.$$
This yields
$$u_{n+1}=u_{2k}\le u_{k}+u_{k+1}\le v_{k}+v_{k+1}=v_{2k}=v_{n+1},$$
where the last inequality uses the induction hypothesis.

\medskip
$\bullet$ If $n+1$ is odd then $n+2$ is even, say $n+2=2k$. In that case $n+2\ge 6$ which yields $k\ge 3$. Hence
$$k=n+2-k\le n-1.$$
We can thus apply he induction hypothesis to get
$$u_{n+1}\le u_{n+2}-1=u_{2k}\le u_{k}+u_{k+1}-1\le v_{k}+v_{k+1}-1=v_{2k}-1=v_{n+2}-1=v_{n+1}.$$
\end{proof}

\begin{lemma}
\label{lem-v4nrecu}
The sequence $v_{n}$ satisfies for every $n\ge 2$,
$$\forall n\ge 2,\begin{cases}
v_{4n}=v_{2n}+v_{2n+2}-1,\\
\quad v_{4n+2}=2v_{2n+2}-1.
\end{cases}
$$
\end{lemma}
\begin{proof}
For $n\ge 2$,  $v_{4n}=v_{2n}+v_{2n+1}$ and $v_{2n+1}=v_{2n+2}-1$.
For $n\ge 2$, $v_{4n+2}=v_{2n+1}+v_{2n+2}$ and $v_{2n+1}=v_{2n+2}-1$.
\end{proof}

\begin{lemma}
\label{lem-vnquasiarithm}
For every $n\ge1$, $v_{2n+2}-v_{2n}\in\{\be,\be+\al-1\}$.
\end{lemma}
\begin{proof}
A simple computation shows that this holds up to $v_{22}-v_{20}$. We thus do the proof by induction.
Assume $n\ge 4$ is such that for any $p\le n$, $v_{2p+2}-v_{2p}$ is either equal to $\be$ or $\be+\al-1$.
Hence we have
$$v_{2n+4}-v_{2n+2}=v_{n+2}+v_{n+3}-v_{n+1}-v_{n+2}=v_{n+3}-v_{n+1}.$$
If $n$ is even, then $5\le n+1\le n+3$ are odd and equalities $v_{n+3}=v_{n+4}-1$  $v_{n+1}=v_{n+2}-1$ hold. Hence we get
$$v_{2n+4}-v_{2n+2}=v_{n+4}-v_{n+2},$$
with $n+2\le 2n$. Hence induction hypothesis applies and $v_{2(n+1)+2}-v_{2(n+1)}$ is either equal to $\be$ or to $\be+\al-1$.

\medskip
If $n$ is odd, $5\le n+1\le n+3$ are even and $n+1\le 2n$ since $n\ge 4$. This also yields that $v_{2(n+1)+2}-v_{2(n+1)}$ is either equal to $\be$ or to $\be+\al-1$.
\end{proof}

Lemma \ref{lem-vnquasiarithm} yields that for every $n$, $v_{2n}\le (\be+\al-1)(n-1)+\al$.

\subsubsection{Bound from above for $\kappa_{n}$}

We remind that any even line in $\GJ$ is a concatenation of $0010$ or $0000$. Any odd line is a concatenation of $10$. This yields that  only the configurations \tritree{0}{0}{0}, \tritree{1}{0}{0} , \tritree{0}{1}{0} , \tritree{1}{1}{0}  appear in $\GJ$. Hence $\kappa_{2}=4$. Set $\al=4$ and $\be=\kappa_{3}$.
Let us consider the sequences $(u_{n})=(\kappa_{n})$ and $(v_{n})$ as above.

Lemma \ref{lem-unvn} yields  for all $n\ge 2$
$$\kappa_{2n}\le (\be+3)(n-1)+4, \kappa_{2n+1}\le \kappa_{2n+2}-1=(\be+3)n+3.$$

\section{Bufetov entropy} \label{secbufetov-entropy}

First, we recall Bufetov's definition of entropy in our settings.


We write $\psi \leq \omega$ when
there exists some $\eta$ such that $ \omega = \eta \psi$ and define the dynamical
distance $d_{\omega}$ on $X$ by
$$
   d_{\omega}(\GA, \GB) = \max_{\psi \leq \omega}{ d(T_{\psi}(\GA), T_{\psi}(\GB))  }.
$$
For $\epsilon>0$, an $(\omega, \epsilon, T_a, T_b)-$separated set is a set $K \subset X$ such that
for any pair $\GA, \GB \in K$, with $ \GA\neq \GB$, $d_{\omega}(\GA, \GB) \geq \epsilon$. The maximal cardinality of an
$(\omega, \epsilon, T_a, T_b)-$separated set is then denoted by $N(\omega, \epsilon, T_a, T_b)$.

Now take
$$
  N(n, \epsilon, T_a, T_b) = \frac{1}{2^n} \sum_{|\omega|=n} N(\omega, \epsilon, T_a, T_b)
$$

The Bufetov entropy of the action is defined as
$$
 h_B(T) = \lim_{\epsilon\to0}\limsup_{\ninf}\frac1n\log(N(n,\epsilon, T_a, T_b)).
$$

Setting $Y:= \{a, b\}^{\NN} \times X $, 
$$
  F: (\omega, \GA) \in  \{a, b\}^{\NN} \times X \mapsto (\sigma(\omega), T_{\omega_0}(\GA))
$$
where $\sigma$ is the usual shift, and $$
 d_{Y}((\om,\GA);(\om',\GA'))=\max(d_{\S}(\om,\om');d(\GA,\GA')), 
$$
Bufetov proved in \cite{bufetov} equality 
$$h_{top}(F)=h_{B}(T)+\log2.$$

%
%
%

\subsection{Proof of Theorem \ref{main-Fzeroentro}}

By definition (see \cite{walters}),
$$h_{top}(F):=\lim_{\epsilon\to0}\limsup_{\ninf}\frac1n\log(r(n,\epsilon)),$$
where $r(n,\epsilon)$ is the maximal cardinality for a $(n,\epsilon)$-separated set of points (for the metric $d_{Y}$).

\bigskip
Note that $\max(c,c')\le\epsilon$ is equivalent to $c\le \epsilon$ and $c'\le \epsilon$.
Pick $\epsilon:=2^{-p}$.

We remind that $k_{n}$ denotes the cardinality of the set $K_{n}$ of patches of length $n$. Because $(X,\FFF)$ is expanding,
$k_{n+p}=\#K_{n+p}$ is the maximal cardinality for an $(n,\epsilon)$-separated set in $X$.
Similarly, $(\S_{2},\s)$ is expanding and $2^{n+p}$ is the maximal cardinality for a $(n,\epsilon)$-separated set in $\S$.
This yields (for sufficiently large $n$)
$$r(n,\epsilon)\le 2^{n+p}k_{n+p}\le 2^{n+p}C(n+p)$$
where we use the bound from above for $\kappa_{n}$. 
This yields $h_{top}(F)\le \log 2$.

On the other hand, for any $\al_{1}\ldots \al_{n}\in \{a,b\}^{n}$ any maximal $(n,\epsilon)$-separated set in $Y$ must  contain a point $(\om,\GA)$ with $\om_{0}\ldots \om_{n-1}=\al_{1}\ldots \al_{n}$. Hence $r(n,\epsilon)\ge 2^{n+p}$. This yields $h_{top}(F)\ge \log 2$.
Hence, $h_{top}(F) = \log{2}$ and $h_B(T) =0$ as claimed.

\end{document}